\documentclass[a4paper]{article}

\usepackage{amsmath}
\usepackage{amssymb}
\usepackage{dsfont}
\usepackage{ifthen}
\usepackage{xparse}

\newcommand{\compl}[1]{#1^c} 
\newcommand{\compsptc}[1]{C_c^{#1}} 
\renewcommand{\div}{\mathrm{div}\ } 

\renewcommand{\epsilon}{\varepsilon}

\let\foralltemp\forall
\renewcommand{\forall}{\foralltemp\, }

\newcommand{\grad}{\mathrm{grad}\,}
\newcommand{\haumeas}[2][]{\mathcal{H}_{#1}^{#2}} 
\newcommand{\indicator}[1]{\mathds{1}_{#1}} 
\DeclareDocumentCommand{\integral}{ O{} O{} m O{x}}{\int_{#1}^{#2} #3\, \mathrm{d}#4} 

\newcommand{\Natural}{\mathbb{N}}
\newcommand{\norm}[2][]{\| #2 \|_{#1}} 

\newcommand{\Real}{\mathbb{R}} 

\newcommand{\set}[2]{\left\{ #1 : #2 \right\}} 
\newcommand{\setdiff}{\backslash} 
\newcommand{\simpleset}[1]{\left\{ #1 \right\}} 
\newcommand{\support}{\mathrm{spt}\ } 
\newcommand{\unitsphere}[1]{\mathbb{S}^{#1-1}}

\newcounter{object}[section]
\renewcommand{\theobject}{\arabic{section}.\arabic{object}}

\newenvironment{corollary}[1][]{\medskip%
			 \refstepcounter{object}%
			 \noindent \textbf{Corollary \theobject%
			 \ifthenelse{\equal{#1}{}}{.}{\ (#1).}\ }\itshape }{\medskip}

\newenvironment{lemma}{\rmfamily \medskip%
			  \refstepcounter{object}%
			  \noindent \textbf{Lemma \theobject .}\ \itshape}{}
\newenvironment{proposition}[1][]{\rmfamily \medskip%
			  \refstepcounter{object}%
			  \noindent \textbf{Proposition \theobject%
			 \ifthenelse{\equal{#1}{}}{.}{\ (#1).}\ }\itshape }{}
\newenvironment{proof}[1][]{
  \noindent \textit{Proof\ifthenelse{\equal{#1}{}}{}{ of #1}.}\ }{\hfill $\blacksquare$}
\newenvironment{theorem}[1][]{\medskip%
			 \refstepcounter{object}%
			 \noindent \textbf{Theorem \theobject%
			 \ifthenelse{\equal{#1}{}}{.}{\ (#1).}\ }\itshape }{}

\usepackage{fullpage}
\usepackage{color}

\newcommand{\gab}{g_{\alpha \beta}}
\newcommand{\gabhat}{\hat g_{\alpha \beta}}

\title{Fractional Sobolev norms and BV functions on manifolds}
\author{Andreas Kreuml\footnote{partially supported by FWF project I3027-N35} , Olaf Mordhorst}
\date{}

\begin{document}

\maketitle

\begin{abstract}
	The bounded variation seminorm and the Sobolev seminorm on compact manifolds are represented as a limit of fractional Sobolev seminorms.
	This establishes a characterization of functions of bounded variation and of Sobolev functions on compact manifolds.
	As an application the special case of sets of finite perimeter is considered.
\end{abstract}



\section{Introduction and main results}

In the early 2000's the study of fractional $s$-seminorms gained new interest, when Maz'ya \& Shaposhnikova \cite{mazya} on one hand, and Bourgain, Brezis \& Mironescu \cite{bbm} on the other hand showed, that they can be seen as intermediary functionals between the $L^1$-norm and the $W^{1,p}$-seminorms.
For $\Omega \subseteq \Real^n$ open, $0 < s < 1$ and $1 \le p < \infty$, Gagliardo \cite{gagliardo} introduced the fractional Sobolev space $W^{s,p}(\Omega)$ as the set of all functions $f \in L^p(\Omega)$ such that the seminorm
\begin{equation*}
	|f|_{W^{s,p}} := \left( \integral[\Omega]{
			\integral[\Omega]{
				\frac{|f(x)-f(y)|^p}{|x-y|^{n+sp}}
			}[y]
	} \right)^{\frac 1 p}
\end{equation*}
is finite. We say that \(f\) is in the Sobolev space \(W^{1,p}(\Omega)\) if the weak gradient \(\nabla f\) exists and such that the seminorm \(|f|_{W^{1,p}} := \left(\integral[\Omega]{|\nabla f(x)|^p} \right)^{\frac 1 p}\) is finite. If \(p=1\) we say that  \(f\) is a function of bounded variation if the  seminorm
\begin{equation*}
	|Df|(\Omega) := \sup \set{\integral[\Omega]{f\, \div T}}{T \in \compsptc 1 (\Omega;\Real^n), |T(x)| \le 1\ \text{for all}\ x \in \Omega},
\end{equation*}
is finite and we denote by \(BV(\Omega)\) the space of functions of bounded variation.
Here, $\compsptc 1 (\Omega; \Real^n)$ denotes the set of all continuously differentiable functions $T: \Omega \to \Real^n$ such that the support is compact in $\Omega$.
We use the convention that \(|f|_{W^{1,p}}=\infty\) if \(f\not\in W^{1,p}(\Omega)\) and \(|Df|(\Omega)=\infty\) if \(f\not\in BV(\Omega)\). 

It is quite natural to ask if fractional differentiability for every order strictly between \(0\) and \(1\) implies differentiability of order \(1\) in the Sobolev sense. Indeed, Bourgain, Brezis \& Mironescu proved for \(1<p<\infty\) and \(\Omega\) smooth and bounded that
for every \(f\in L^p(\Omega)\) we have
\begin{equation}
\label{bbmlp}
\lim_{s\rightarrow 1^-}(1-s)\integral[\Omega]{
			\integral[\Omega]{
				\frac{|f(x)-f(y)|^p}{|x-y|^{n+sp}}
			}[y]
		}= \alpha_{n,p} |f|_{W^{1,p}}^p\ ,
\end{equation}
where $\alpha_{n,p}$ is a constant only depending on \(n\) and \(p\) (see \cite{bbm}, Corollary 2). This convergence result is false in general if \(p=1\) since the class \(W^{1,1}(\Omega)\) is simply too small in this case. This led Bourgain, Brezis and Mironescu to consider the class \(BV(\Omega)\): A function \(f\in L^1(\Omega)\) is in \(BV(\Omega)\) if and only if 
\begin{equation}
\label{bbmchar}
\liminf_{s \to 1}\ (1-s)\integral[\Omega]{
			\integral[\Omega]{
				\frac{|f(x)-f(y)|}{|x-y|^{n+s}}
			}[y]
		}<\infty\ .
\end{equation}
The question of the convergence of the \(|\cdot|_{W^{s,1}}\)-seminorms was answered by D\'{a}vila in \cite{davila} who proved that
\begin{equation}
	\label{bbmsemi}
	\lim_{s \to 1^-} (1-s)\integral[\Omega]{
			\integral[\Omega]{
				\frac{|f(x)-f(y)|}{|x-y|^{n+s}}
			}[y]
	} = 2 |B^{n-1}| |Df|(\Omega)
\end{equation}
for bounded \(\Omega\) with Lipschitz-boundary and \(f\in BV(\Omega)\). 
Here, $B^k$ denotes the $k$-dimensional Euclidean unit ball and $|B^k|$ its Lebesgue measure of corresponding dimension.
By a counterexample of Brezis \cite{brezis}, the  results \eqref{bbmlp}, \eqref{bbmchar} and \eqref{bbmsemi} fail to hold in general on non-smooth open sets $\Omega$.

Still, Leoni \& Spector \cite{leonispector} recovered a variant of \eqref{bbmsemi} for arbitrary open sets.

Since then, many related questions and generalizations were studied, ranging from classification results for Sobolev and BV spaces (\cite{brezis}, \cite{leonispector}), anisotropic higher order Sobolev spaces (\cite{pinamonti}), fractional perimeters (\cite{crs}, \cite{cesaroni}), sharp fractional Sobolev and isoperimetric inequalities (\cite{dicastro}, \cite{frankseiringer}, \cite{hurri}, \cite{xiao}),  to anisotropic versions of fractional seminorms and perimeters (\cite{ludwigperimeter}, \cite{ludwignorm}, \cite{xiaoye}).
Fractional seminorms and BV functions cannot only be defined on open subsets of $\Real^n$, but also on Riemannian manifolds or metric measure spaces (\cite{dimarinosquassi}).
Independently, the authors of \cite{carbonaro} and \cite{miranda} showed, that the variation $|Df|(M)$ of a BV function on a manifold can be approximated by evolutions of the function under the heat semigroup. In \cite{ferrari}
a characterization of perimeters in Carnot groups is provided via heat semigroup techniques. The authors raise the question if a characterization of perimeters can also be attained on Riemannian manifolds.

The purpose of this paper is to generalize \eqref{bbmlp}-\eqref{bbmsemi} to the setting of compact Riemannian manifolds. As for example in the papers \cite{bbm}, \cite{davila} and \cite{leonispector} we prove our result in the following slightly more general framework from which the desired convergence of the fractional Sobolev norms follows as a corollary. 

Let $\rho_\sigma: (0,\infty) \to [0,\infty), 1>\sigma > 0$ be a family of functions. These functions are called  \emph{radial mollifiers} if they satisfy the following properties:
\begin{align}
	\label{rhomon} \rho_\sigma\ \text{is monotonically } & \text{decreasing on}\ (0,\infty),\\
	\label{eq:fixedmass}
	\integral[0][\infty]{\rho_\sigma(r) r^{n-1}}[r] & = \frac 1 {\haumeas{n-1}(\unitsphere n)} \quad \text{for all }\sigma ,\\
	\label{eq:tailto0}
	\lim_{\sigma \to 0} \integral[\delta][\infty]{\rho_\sigma(r) r^{n-1}}[r] & = 0, \quad \forall \delta > 0.\\
	\label{rholocuni} \lim_{\sigma \to 0} \sup_{r \in K} \rho_\sigma(r) &= 0,\quad \forall K \subset (0,\infty)\ \text{compact}.
\end{align}
Here, $\haumeas{n-1}$ denotes the $(n-1)$-dimensional Hausdorff measure and $\unitsphere n := \set{x \in \Real^n}{|x| = 1}$ is the Euclidean unit sphere.
We consider a Riemannian manifold $M$ with metric $g$, and denote by $d(\cdot,\cdot)$ the geodesic distance on $M$ and by $\mathrm{d}V_g$ the Riemannian volume form.
Using kernels satisfying \eqref{rhomon}-\eqref{rholocuni} we show:

\begin{theorem}
	\label{convergencerho}
	Let $M$ be a compact connected $n$-dimensional Riemannian manifold and $f \in L^p(M)$ with $p \ge 1$.
	Furthermore, let $(\rho_\sigma)_{\sigma}$ be a family of radial mollifiers.
	\begin{enumerate}
		\item
			If $p > 1$, then 
	\begin{equation*}
		\lim_{\sigma \to 0} \integral[M]{
			\integral[M]{
				\frac{|f(x)-f(y)|^p}{d(x,y)^p} \rho_\sigma(d(x,y))
			}[V_g(y)]
		}[V_g(x)] = K_{p,n} \integral[M]{|\grad f(x)|_g^p}[V_g(x)],
	\end{equation*}
	where the constant $K_{p,n}$ is defined as
	\begin{equation}
		\label{kpn}
		K_{p,n} := \frac{1}{\haumeas{n-1}(\unitsphere n)} \integral[\unitsphere n]{|e \cdot u|^p}[\haumeas{n-1}(u)],
	\end{equation}
	and $e \in \unitsphere n$ is any unit vector.
	In particular, $f \in W^{1,p}(\Omega)$ if and only if
	\begin{equation*}
		\liminf_{\sigma \to 0} \integral[M]{
			\integral[M]{
				\frac{|f(x)-f(y)|^p}{d(x,y)^p} \rho_\sigma(d(x,y))
			}[V_g(y)]
		}[V_g(x)] < \infty. 
	\end{equation*}
\item
	If $p = 1$, then
	\begin{equation*}
		\lim_{\sigma \to 0} \integral[M]{
			\integral[M]{
				\frac{|f(x)-f(y)|}{d(x,y)} \rho_\sigma(d(x,y))
			}[V_g(y)]
		}[V_g(x)] = K_{1,n} |Df|(M),
	\end{equation*}
	with the constant $K_{1,n}$ defined in \eqref{kpn}.
	In particular, $f \in BV(M)$ if and only if
	\begin{equation*}
		\liminf_{\sigma \to 0} \integral[M]{
			\integral[M]{
				\frac{|f(x)-f(y)|}{d(x,y)} \rho_\sigma(d(x,y))
			}[V_g(y)]
		}[V_g(x)] < \infty.
	\end{equation*}
	\end{enumerate}
\end{theorem}
We define the function space $BV(M)$ of BV functions on a Riemannian manifold in Section \ref{notandback} but it is almost the same as in the Euclidean case. The condition (\ref{rhomon}) seems unnecessary but we impose it for technical reasons since our proofs would be less clear otherwise and since our main application is the case of fractional Sobolev norms. The case of not connected manifolds means that \(d(x,y)=\infty\) for \(x,y\) of different connected components and the integrand is interpreted as \(0\) for such \(x,y\). Hence, this just leads to work on the connected components of \(M\) separately which does not bring anything new to the problem. 

As a simple consequence of Theorem \ref{convergencerho} the convergence of $s$-seminorms follows:

\begin{corollary}
	\label{convergencesemi}
	Let $M$ be a compact connected $n$-dimensional Riemannian manifold and $f \in L^p(M), {p \ge 1}$.
	Then
	\begin{equation*}
		\lim_{s \to 1^-} (1-s) \integral[M]{
			\integral[M]{
				\frac{|f(x)-f(y)|^p}{d(x,y)^{n+sp}}
			}[V_g(y)]
		}[V_g(x)] = \frac{\haumeas{n-1}(\unitsphere n) K_{p,n}}{p} \integral[M]{|\grad f(x)|_g^p}[V_g(x)]\ ,
	\end{equation*}
	if $p > 1$, and
	\begin{equation*}
		\lim_{s \to 1^-} (1-s) \integral[M]{
			\integral[M]{
				\frac{|f(x)-f(y)|}{d(x,y)^{n+s}}
			}[V_g(y)]
		}[V_g(x)] = 2 |B^{n-1}| |Df|(M)\ ,
	\end{equation*}
	if $p = 1$.
\end{corollary}

The double integrals on the left-hand side correspond to the fractional Sobolev norms on manifolds in the sense of Gagliardo. Of course, this notion does only make sense in the compact case or if there are at least some bounds on the volume growth of the manifold. 
 
The theory of fractional seminorms can be applied to study the size of the boundary for a large class of sets, leading to the notion of fractional perimeters.
Corollary \ref{convergencesemi} in particular implies that fractional perimeters converge to the perimeter as $s \to 1^-$ up to a constant.
We will discuss both notions in Section \ref{notandback} in further detail.

\section{Notation and background material}
\label{notandback}

Throughout this paper we denote by $M$ a compact connected $n$-dimensional Riemannian manifold of class $C^\infty$.
We denote its metric by $g$ and the function $|V|_g := \sqrt{g(V,V)}$ defines a norm on each tangent space.
If $(U,\phi)$ is a chart on $M$ and $f: M \to \Real$ is a function, then we put $\hat f: \phi(U) \to \Real$ for its coordinate representation.
Furthermore, we write $\gab$ for the components of the metric $g$ with respect to a given chart, i.e. $\gab := g(\frac \partial {\partial x^\alpha}, \frac \partial {\partial x^\beta})$, where $\frac \partial {\partial x^\alpha}, \frac \partial {\partial x^\beta}$ are elements of the tangential bundle of $M$, which we denote by $TM$.
The Riemannian volume form $\mathrm{d}V_g$ gives rise to a measure $\mathrm{Vol}_g$ on $M$ defined by $\mathrm{Vol}_g (A) := \integral[A]{}[V_g]$, $A \subseteq M$ Borel.
We write 
\begin{equation*}
	\support f : = \overline{\set{x \in M}{f(x) \neq 0}}
\end{equation*}
for the support of a function $f: M \to V$, where $V$ is a vector space or bundle, and if $\mathcal F$ is a function space, then $\mathcal F_c$ denotes the subset of $\mathcal F$ consisting of all compactly supported functions.
Open balls with center $x$ and radius $r$ are commonly denoted by $B_r^M(x)$ for geodesic balls on $M$ and $B_r^k(x)$ for Euclidean balls in $\Real^k$.
We further write $B_r^k := B_r^k (0)$.
We define the indicator function of a set $E \subseteq M$ as
\begin{equation*}
	\indicator{E}(x) := \begin{cases}
		1, & x \in E,\\
		0, & x \notin E
	\end{cases}\quad .
\end{equation*}

If $f : M \to \Real$ is a smooth function, we define the gradient of $f$ as the smooth vector field $\grad f$ satisfying $g(\grad f, X) = df(X)$ for all smooth vector fields $X$ on $M$.
The symbol $\nabla$ is used exclusively to denote (weak) gradients in $\Real^n$.
The divergence $\div X$ of a smooth vector field $X$ on $M$ is defined as the Lie derivative of $\mathrm{d}V_g$ with respect to $X$, i.e. $\div X = L_X (\mathrm{d}V_g)$.
In local coordinates with respect to a chart $\phi$, the vector field $X = \sum_{i=1}^n X^i \frac \partial {\partial x^i}$ gives rise to a vector field $T = (X^1 \circ \phi^{-1}, \dots, X^n \circ \phi^{-1})^T$ in $\Real^n$ and the divergence of $X$ can be expressed as
\begin{equation}
	\label{divchart}
	(\div X)(x) = \frac 1 {\sqrt{\det (\gab(x))} } \text{div}_{\Real^n}\left(\sqrt{\det (\gabhat(\cdot))}\, T \right)(\phi(x)),
\end{equation}
where $\text{div}_{\Real^n}$ denotes the divergence operator in $\Real^n$.
As in the Euclidean setting, we define the weak gradient of a function $f \in L^1(M)$ as the unique vector field $Y$ on $M$, such that $\integral[M]{|Y|_g}[V_g] < \infty$ and for all smooth vector fields $X$ on $M$
\begin{equation*}
	\integral[M]{g(Y,X)}[V_g] = - \integral[M]{f\ \div X}[V_g]
\end{equation*}
holds.
Here, uniqueness is understood up to sets of measure zero.
We denote it by $\grad f$ and justify this notation by remarking that for smooth functions the (standard) gradient and the weak gradient coincide.
For a more detailed discussion on differential operators on Riemannian manifolds we refer to \cite{lee}.

For $1 \le p < \infty$ we define the Sobolev space $W^{1,p}(M)$ by
\begin{equation*}
	W^{1,p}(M) := \set{f \in L^p (M)}{\text{the weak gradient } \grad f \text{ exists and } |\grad f|_g \in L^p(M)}.
\end{equation*}
Equipped with the norm
\begin{equation*}
\norm[W^{1,p}]{f} := \left( \norm[L^p]{f}^p + \integral[M]{|\grad f|_g^p}[V_g] \right)^{\frac 1 p},
\end{equation*}
$W^{1,p}(M)$ is a Banach space (c.f. \cite[p. 21]{hebey}).

We need an alternative characterization of Sobolev spaces, which works only if $p > 1$, since only in this case the spaces $L^p(M)$ and $W^{1,p}(M)$ are reflexive (c.f. \cite[Prop. 2.3]{hebey}):

\begin{proposition}
	\label{sobapprox}
	Let $f \in L^p(M),\, p > 1$. Then $f \in W^{1,p}(M)$ if and only if there exists a sequence $(f_j)_j \subset \compsptc \infty (M)$ such that the following two statements hold:
	\begin{enumerate}
		\item $f_j \stackrel {j \to \infty} \to f$ in $L^p(M)$, and
		\item $\displaystyle L := \lim_{j \to \infty} \integral[M]{|\grad f_j|_g^p}[V_g] < \infty$.
	\end{enumerate}
	In this case $L = \integral[M]{|\grad f|_g^p}[V_g]$.
\end{proposition}
\\

In analogy to the Euclidean case, the \emph{variation} of a function $f \in L^1(M)$ is introduced in \cite{miranda} as a measure given on open sets $U \subseteq M$ by
\begin{equation}
	\label{variation}
	|Df|(U) := \sup \set{\integral[M]{f\, \div X}[V_g]}{X \in \Gamma_c (TM), \support X \subset U, |X(x)|_g \le 1\ \text{for all}\ x \in M},
\end{equation}
where $\Gamma_c (TM)$ denotes the space of all compactly supported vector fields of class $C^\infty$. 
The definition works also for not necessarily compact manifolds but since we only work on compact \(M\), the condition that the vector fields involved are compactly supported can of course be dropped.
We say that $f$ is of \emph{bounded variation} and write $f \in BV(M)$, if $|Df|(M) < \infty$. 
For an exhaustive discussion of $BV$ functions in the Euclidean setting, we refer to \cite{ambrosio-bv}.

If $f \in C^\infty(M)$, then
\begin{equation*}
	|Df|(U) = \integral[U]{|\grad f|_g}[V_g]
\end{equation*}
for all open $U \subseteq M$.
This can be seen as follows:
Since $M$ is a manifold without boundary, the divergence theorem implies
\begin{equation*}
	0 = \integral[M]{\div\!\!(fX)}[V_g] = \integral[M]{f\, \div X}[V_g] + \integral[M]{g(\grad f, X)}[V_g],
\end{equation*}
for every smooth vector field $X \in \Gamma_c (TM)$.
Thus, we can approximate the supremum in \eqref{variation} by a sequence of smooth vector fields converging to $-\indicator {\simpleset{\grad f \neq 0}} \frac{\grad f}{|\grad f|_g}$.

A related concept is the notion of \emph{weighted BV functions}, as introduced in \cite{baldi} for the Euclidean case.
Let $\Omega \subseteq \Real^n$ open and $\Omega_0$ an open neighbourhood of $\overline \Omega$.
We call a lower semicontinuous function $w \in L_{loc}^1(\Omega_0), w > 0$, satisfying
\begin{equation*}
	\frac 1 {|B_r^n(x)|} \integral[B_r^n(x)]{w(y)}[y] \le Cw(x)
\end{equation*}
for all balls $B_r^n(x) \subset \Omega_0$ with a constant $C > 0$, a weight.
The variation of a function $f \in L^1(\Omega; w\, \mathrm d x)$ with respect to the weight $w$ is defined as
\begin{equation*}
	|Df|_w (\Omega) := \sup \set{\integral[\Omega]{f\, \div T}}{T \in C_c^1(\Omega; \Real^n), |T(x)| \le w(x) \text{ for all } x \in \Omega},
\end{equation*}
and the space $BV(\Omega;w)$ consists of those functions $f$ such that $|Df|_w (\Omega) < \infty$.
In accordance to the case of unweighted BV functions, the map $f \mapsto |Df|_w (\Omega), f \in BV(\Omega;w),$ is lower semicontinuous with respect to $L^1 (\Omega;w\,\mathrm{d}x)$-convergence, see \cite[Theorem 3.2]{baldi}.

The following lemma establishes a link between the notions of variation on a manifold and weighted variation in $\Real^n$, as well as an analogous result for weak gradients.
A short proof of the second statement was given in \cite{miranda}.
Some arguments of the proof are not accessible to us, so we include an alternative proof.

\begin{lemma}
	\label{diffchart}
	Let $\phi: U \to \Real^n$ be a chart on $M$ such that the operator norm of $d\phi|_x : (T_x M, |\cdot|_g) \to (\Real^n, |\cdot|)$ satisfies $\norm{d\phi|_x} \le 1 + \epsilon$ for all $x \in U$.

	\begin{enumerate}
		\item If $f \in W^{1,p}(U)$, then
	for all $\xi \in \phi(U)$
	\begin{equation} \label{weakgradphi}
		|\grad f(\phi^{-1}(\xi))|_g \le (1+\epsilon) |\nabla (f \circ \phi^{-1})(\xi)|\ .
	\end{equation}

\item If $f \in BV(U)$, then
	\begin{equation} \label{varphi}
		|Df|(U) \le (1+\epsilon) |D(f \circ \phi^{-1})|_w (\phi(U))
	\end{equation}
	with weight $w = \sqrt{\det (\hat g_{\alpha \beta})}$.
	\end{enumerate}
\end{lemma}

\begin{proof}
	Let $f \in W^{1,p}(U)$.
	Since $\nabla(f \circ \phi^{-1})$ is the weak gradient of $f \circ \phi^{-1}$, we have for all smooth compactly supported vector fields $T \in \compsptc \infty (\phi(U);\Real^n)$:
	\begin{align*}
		- \integral[\phi(U)]{
			\nabla (f \circ \phi^{-1})(\xi) \cdot T(\xi)
		}[\xi] & =  \integral[\phi(U)]{
			(f \circ \phi^{-1})(\xi)\, \text{div}_{\Real^n} T(\xi)
		}[\xi] \\
		& =  \integral[\phi(U)]{
			f (\phi^{-1}(\xi))\, \text{div}_{\Real^n} \left(\sqrt{\det(\hat g_{\alpha \beta})} \frac T {\sqrt{\det(\hat g_{\alpha \beta})}} \right) (\xi)
		}[\xi].
	\end{align*}
	If $\xi = \phi(x)$ and $T(\xi) = (T^1(\xi), \dots, T^n(\xi))^T \in \Real^n \cong T_\xi \phi(U)$, then $d(\phi^{-1})|_\xi (T(\xi)) = \sum_{i=1}^n T^i(\phi(x)) \frac \partial {\partial x^i}|_x$ with respect to the chart $\phi$ , so $d(\phi^{-1})$ is a bijection between vector fields on $\phi(U)$ and vector fields on $U$. 
	We put $X(x) := d(\phi^{-1})|_{\xi}(T(\xi))$ and use the representation \eqref{divchart} of the divergence in coordinates to obtain
	\begin{align*}
  & \integral[\phi(U)]{
			 f (\phi^{-1}(\xi))\, \text{div}_{\Real^n}   \left(\sqrt{\det(\hat g_{\alpha \beta})} \frac T {\sqrt{\det(\hat g_{\alpha \beta})}} \right) (\xi)
		}[\xi] \\
		 =& \integral[U]{ 
			f(x) \sqrt{\det(g_{\alpha \beta}(x))}\ \div \left( \frac X {\sqrt{\det(g_{\alpha \beta})}} \right)(x) \frac 1 {\sqrt{\det(g_{\alpha \beta}(x))}}
		}[V_g(x)] \\
		 = & - \integral[U]{
			g( \grad f(x), X(x) ) \frac 1 {\sqrt{\det(g_{\alpha \beta}(x))}}
		}[V_g(x)].
	\end{align*}
	In analogy to the differential of a smooth function, we denote the by $df|_x$ the covector field $X(x) \mapsto g( \grad f(x), X(x) ), X(x) \in T_x M$ and further rewrite the last integral as
	\begin{align*}
		- \integral[U]{
			df|_x (X(x)) \frac 1 {\sqrt{\det(g_{\alpha \beta}(x))}}
		}[V_g(x)] = - \integral[\phi(U)]{
			df|_{\phi(\xi)}(d(\phi^{-1})|_\xi (T(\xi)))
			}[\xi],
	\end{align*}
	which proves the chain rule $\nabla(f \circ \phi^{-1})(\xi) = (df|_{\phi(\xi)} \circ d(\phi^{-1})|_\xi)^T$ for weak gradients. 
	It is equivalent to $d(f \circ \phi^{-1})|_{\xi} \circ d\phi|_{\phi^{-1}(\xi)} = df|_{\phi^{-1}(\xi)}$, so by duality we obtain the estimate
	\begin{align*}
	|\grad f(\phi^{-1}(\xi))|_g &= \norm{d(f \circ \phi^{-1})|_{\xi} \circ d\phi|_{\phi^{-1}(\xi)}} \\
	&\le \norm{d(f \circ \phi^{-1})|_{\xi}} \cdot \norm{d\phi|_{\phi^{-1}(\xi)}} \le (1+\epsilon) |\nabla(f \circ \phi^{-1})(\xi)|,
	\end{align*}
	which shows the first statement.

	If $f \in BV(U)$ and if $X$ is a compactly supported vector field in $U$ with $|X(x)|_g \le 1$ for all $x \in U$, then the vector field $T$ on $\phi(U)$ defined by $T(\xi) := d\phi|_{\phi^{-1}(\xi)}(X(\phi^{-1}(\xi)))$ is smooth, compactly supported and satisfies the inequality $|T(\xi)| \le 1+\epsilon$, since
	\begin{align*}
		|T(\xi)| & = |(d\phi|_{\phi^{-1}(\xi)} \circ d(\phi^{-1})|_\xi) (T(\xi))| \\
		& \le (1+\epsilon) |X(\phi^{-1}(\xi))|_g \le 1 + \epsilon.
	\end{align*}
	We apply formula \eqref{divchart} for the divergence in coordinates and compute 
	\begin{equation*}
		\integral[U]{f\ \div X}[V_g] = \integral[\phi(U)]{(f \circ \phi^{-1})\ \text{div}_{\Real^n}\left( \sqrt{\det(\hat g_{\alpha \beta}}) T \right)}[\xi].
\end{equation*}
Thus,
	\begin{align*}
		|Df| & (U) = \sup \set{\integral[U]{f\ \div X}[V_g]}{X \in \Gamma_c(TM), \support X \subset U, |X(x)|_g \le 1\ \forall x \in M} \\
		&\le \sup \set{\integral[\phi(U)]{(f \circ \phi^{-1})\ \text{div}_{\Real^n}\left( \sqrt{\det(\hat g_{\alpha \beta}}) T \right)}[\xi]}{T \in \compsptc \infty (\phi(U);\Real^n), |T(\xi)| \le 1+\epsilon\ \forall \xi \in \phi(U)} \\
		&\le (1+\epsilon)\sup \set{\integral[\phi(U)]{(f \circ \phi^{-1})\ \text{div}_{\Real^n} \tilde T }[\xi]}{\tilde T \in \compsptc \infty (\phi(U);\Real^n), |\tilde T(\xi)| \le \sqrt{\det(\hat g_{\alpha \beta}(\xi))}\ \forall \xi \in \phi(U)} \\
	& = (1+\epsilon) |D(f \circ \phi^{-1})|_w (\phi(U))
	\end{align*}
	with weight $w = \sqrt{\det(\hat g_{\alpha \beta})}$, which concludes the proof of the second statement.
\end{proof}

The authors in \cite{miranda} used formula \eqref{varphi} to show the following:

\begin{proposition}[\cite{miranda}, Prop. 1.4]
	\label{bvapprox}
	Let $f \in L^1(M)$. Then $f \in BV(M)$ if and only if there exists a sequence $(f_j)_j \subset \compsptc \infty (M)$ such that the following two statements hold:
	\begin{enumerate}
		\item $f_j \stackrel {j \to \infty} \to f$ in $L^1(M)$, and
		\item $\displaystyle L := \lim_{j \to \infty} \integral[M]{|\grad f_j|_g}[V_g] < \infty$.
	\end{enumerate}
	In this case $L = |Df|(M)$ and $\displaystyle \lim_{j \to \infty} \integral[M]{|\grad f_j|_g}[V_g](U) = |Df|(U)$ for every open $U \subseteq M$.
\end{proposition}
\\

The previous proposition provides a different approach to the space of BV function via approximation by smooth functions. 
The authors of \cite{ambrosioghezzi} give even further definitions of BV functions, which all agree on Riemannian manifolds.

%

For special weights $w$, Baldi gave a description of the space $BV(\Omega;w)$:

\begin{proposition}[\cite{baldi}, Prop. 3.5]
	Let $w$ be a Lipschitz continuous weight function on $\Omega$.
	Then a function $f$ belongs to $BV(\Omega;w)$ if and only if $f \in BV(\Omega)$ and $w \in L^1 (d|Df|)$.
	In this case
	\begin{equation}
		\label{weightedvarbyvar}
		|Df|_w (\Omega) = \integral[\Omega]{w}[|Df|].
	\end{equation}
\end{proposition}

The variation of a function can be applied to measure the surface area of a measurable set $E \subseteq M$ in the following way (see e.g. \cite{maggi}):
If the indicator function $\indicator E$ of a set $E \subseteq M$ is of bounded variation, then the \emph{perimeter} of $E$ is defined as $P(E) := |D\indicator E|(M)$.
If the boundary of $E$ is a closed hypersurface of class $C^\infty$ equipped with the metric $\tilde g$ inherited by $M$, then
\begin{equation*}
	P(E) = \text{Vol}_{\tilde g} (\partial E) = \haumeas{n-1}(\partial E),
\end{equation*}
which follows by isometric embedding of $M$ into a Euclidean ambient space of suitable dimension and the result therein (c.f. \cite[Example 12.5]{maggi}).

The Riemannian manifold $M$ carries the geodesic distance, denoted by $d(\cdot,\cdot)$, which allows us to introduce a fractional seminorm for measurable functions $f: M \to \Real$ and $s \in (0,1), 1 \le p < \infty$, as follows:
\begin{equation*}
	|f|_{W^{s,p}} := \left( \integral[M]{
		\integral[M]{
			\frac{|f(x)-f(y)|^p}{d(x,y)^{n+sp}}
		}[V_g(y)]
}[V_g(x)] \right)^{\frac 1 p}\ .
\end{equation*}
For a thorough introduction to fractional seminorms and Sobolev spaces in $\Real^n$ see e.g. \cite{hitchhiker}.

On the other hand, the (fractional) $s$-perimeter of a measurable set $E \subseteq M$, as introduced for subsets of $\Real^n$ in \cite{crs}, can be defined for $s \in (0,1)$ in an analogous way by
\begin{equation*}
	P_s (E) := \integral[E]{
		\integral[\compl E]{
			\frac{1}{d(x,y)^{n+s}}
		}[V_g(y)]
	}[V_g(x)].
\end{equation*}
Computing the fractional seminorm with $p=1$ of the indicator function $\indicator E$ of $E$ yields ${|\indicator E|_{W^{s,p}} = 2 P_s(E)}$.

\section{Proofs}

We define the distance of a point $x \in M$ to a set $E \subseteq M$ by
\begin{equation*}
	d(x,E) := \inf \set{d(x,y)}{y \in E},
\end{equation*}
and for $\tau > 0$ we define the $\tau$-neighbourhood of a set $E \subseteq M$ by
	\begin{equation*}
		E^\tau := \set{x \in M}{d(x,E) < \tau}.
	\end{equation*}
	For our calculations we want to work with families of finitely many open sets in $M$, such that on each set the geodesic distance $d(x,y)$ can be controlled by the Euclidean distance on a corresponding chart (cf. \cite[proof of Prop. 1.4]{miranda}):

\begin{lemma}
	\label{coveringlemma}
	If $E \subseteq M$ is a compact set, then
	for each $\epsilon \in (0,1)$ there exists a finite family $(U_k)_{k=1}^N, N = N(\epsilon),$ of open sets of $M$ such that
	\begin{enumerate}
		\item
			$U_k \cap U_l = \emptyset,\quad \forall k \neq l$,
		\item
			there exists $\tau_0 > 0$ such that for all $0 < \tau < \tau_0$ the family $(U_k^\tau)_{k=1}^{N}$ is an open covering of $E$ and $U_k^{\tau_0}$ lies in the domain of a coordinate chart $(V_k, \phi_k)$, where
			\begin{gather}
				\label{length} (1-\epsilon)|\phi_k(x) - \phi_k(y)| \le d(x,y) \le (1+\epsilon)|\phi_k(x) - \phi_k(y)|,\\
				 \label{detg} 1-\epsilon \le \sqrt{\det(g_{\alpha \beta}(x))} \le 1+\epsilon 
			\end{gather}
			for every $x, y \in V_k$.
		\item \label{opnorm} The operator norm $\norm{d\phi_k|_x}$ of $d\phi_k|_x : (T_x M, |\cdot|_g) \to (\Real^n, |\cdot|)$ is bounded by
			\begin{gather}
				\label{opnorminequality} 1- \epsilon \le \norm{d \phi_k|_x} \le 1 + \epsilon
			\end{gather}
			for every $x \in V_k$.
		\item \label{bdmeasure}
			$\integral[\partial U_k]{}[V_g] = 0$ for all $k = 1,\dots,N$.
	\end{enumerate}
	Furthermore, given a function $f \in BV(M)$, the sets can be chosen in such a way that
	\begin{enumerate}
		\item[4'.] $|Df|(\partial U_k) = 0$ for all $k = 1,\dots,N$.
	\end{enumerate}
\end{lemma}

\begin{proof}
	For each point $p \in E$ there exists a normal coordinate chart $(V_p, \phi_p)$ around $p$ such that the inequalities (\ref{length}) and (\ref{detg}) are satisfied, see e.g. \cite{hanzhu}, p.8. Since the operator norm of \(d\phi_p\) at \(p\) is one we may choose \(V_p\) so small around \(p\)  such that inequality (\ref{opnorminequality}) holds. 
	The compactness of $E$ ensures the existence of $\tau_0 > 0$ such that for every $p$ there exists an open subset $W_p \subset V_p$ around $p$ such that $W_p^{\tau_0} \subseteq V_p$.
	Since the geodesic spheres $\set{y \in E}{d(p,y) = r} \cap W_p$ form a disjoint uncountable covering of $W_p$ and both $\textrm{d}V_g$ and $|Df|$ are finite Radon measures, there exists an open geodesic ball $B_p \subseteq W_p$ such that both $\integral[\partial B_p]{}[V_g] = 0$ and $|Df|(\partial B_p) = 0$ hold.

	By compactness of $E$ there exists an open subcovering $(B_{p_k})_{k=1}^N$ of $(B_p)_{p \in E}$, which can be made disjoint by setting
	\begin{align*}
		U_1 & := B_{p_1},\\
		U_k & := B_{p_k} \setdiff \bigcup_{i=1}^{k-1} \overline U_i,\ k = 2,\dots,N.
	\end{align*}
	Note that the new family does not cover $E$ anymore, but still satisfies conditions \ref{bdmeasure} and 4', because $\partial U_k \subset \bigcup_{i=1}^N \partial B_{p_i}$.

\end{proof}

For the case $p=1$ in the main result we establish that the total variation $|Df|$ of a BV function $f$ on $M$ is a limit of certain integrals. So it is convenient to introduce the following notion, which is also appropriate to use if $p >1$.
For each $\sigma > 0$ and $p \ge 1$ we define the Radon measure $\mu_{\sigma,p}$ on $M$ by
\begin{equation}
	\label{musigma}
	\mu_{\sigma,p}(E) := \integral[E]{
		\integral[M]{
			\frac{|f(x)-f(y)|^p}{d(x,y)^p} \rho_\sigma(d(x,y))
		}[V_g(y)]
	}[V_g(x)],\ E \subseteq M\ \text{Borel}.
\end{equation}
The outline of the proof of our main results follows \cite{davila} and \cite{leonispector}, adapted to the manifold setting.

\begin{proposition}
	\label{fixed-i-estimate}
	Let $E \subseteq M$ be a compact set. 

	If $p >1$ and $f \in W^{1,p}(M)$,
then for every $\epsilon \in (0,1)$ there exist $R_0 > 0$ and a function \(G_{\varepsilon}\) independent of \(\sigma\) such that for every $0 < R < R_0$
	\begin{equation}
		\label{musigmaest}
		\mu_{\sigma,p}(E) \le \frac{(1+\varepsilon)^{p+2}}{(1-\varepsilon)^{p+n}}K_{p,n} \integral[E^{2R}]{|\grad f|_g^p}[V_g] + \frac{\alpha_\sigma}{R^p} \norm[L^p(M)]{f}^p + G_{\epsilon}(R),
	\end{equation}
	where $\lim\limits_{\sigma\rightarrow 0}\alpha_\sigma = 0$ and \(\lim\limits_{R\rightarrow 0}G_{\varepsilon}(R)=0\).

	If $p=1$ and $f \in BV(M)$, then \eqref{musigmaest} holds with $\integral[E^{2R}]{|\grad f|_g^p}[V_g]$ replaced by $|Df|(E^{2R})$.
\end{proposition}
\\
\\
\begin{proof} We may assume without loss of generality that \(\varepsilon<\frac{1}{3}\).
	We divide the proof into two steps:\\
	\\
	{\bf Step 1:} An upper estimate for $\displaystyle \integral[E]{ \integral[d(x,y) < R]{ \frac{|f(x)-f(y)|^p}{d(x,y)^p} \rho_\sigma(d(x,y)) }[V_g(y)] }[V_g(x)]$.\\ 

	First, assume $f \in C^1(M)$.
	Fix $\epsilon \in (0,1/3)$ and let $(U_k)_{k=1}^{N}, N = N(\epsilon)$, be a family of open sets as in Lemma \ref{coveringlemma} and $(V_k, \phi_k)$ the corresponding charts such that $U_k \subseteq V_k$ and (\ref{length}) and (\ref{detg}) hold.
	The following computations are carried out for fixed $k \in \simpleset{1,\dots,N}$, which we will omit for better readability, and with respect to the aforementioned chart.
	Using (\ref{length}) and (\ref{detg}) as well as the monotonicity of $\rho_\sigma$ we have ($\xi := \phi(x), \eta := \phi(y)$)
	\begin{align*}
		& \integral[U \cap E]{
			\integral[B_R^M(x)]{
				 \frac{|f(x)-f(y)|^p}{d(x,y)^p} \rho_\sigma(d(x,y))
			}[V_g(y)]
		}[V_g(x)] \\
		&\le \integral[\phi(U \cap E)]{
			\integral[B_{\frac R {1-\epsilon}}^n(\xi)]{
				\frac{|\hat f(\xi) - \hat f(\eta)|^p}{(1-\epsilon)^p|\xi - \eta|^p} \rho_\sigma((1-\epsilon)|\xi-\eta|) \sqrt{\det \gabhat(\xi)} \sqrt{\det \gabhat(\eta)}
			}[\eta]
		}[\xi] \\
		& \le \frac{(1+\epsilon)^2}{(1-\epsilon)^p} \integral[\phi(U \cap E)]{
			\integral[B_{\frac R {1-\epsilon}}^n(0)]{
				\frac{|\hat f(\xi) - \hat f(\xi + h)|^p}{|h|^p} \rho_\sigma((1-\epsilon)|h|)
				}[h]
			}[\xi] \\
			&\le \frac{(1+\epsilon)^2}{(1-\epsilon)^p} \integral[\phi(U \cap E)]{
				\integral[B_{\frac R {1-\epsilon}}^n(0)]{
					\integral[0][1]{
						\frac{|\nabla \hat f (\xi + th) \cdot h|^p}{|h|^p}
						\rho_\sigma((1-\epsilon)|h|) 
					}[t]
				}[h]
			}[\xi] \\
			& \stackrel{\tilde \xi = \xi + th} \le \frac{(1+\epsilon)^2}{(1-\epsilon)^p} \integral[B_{\frac R {1-\epsilon}}^n(0)]{
				\integral[0][1]{
					\integral[\phi(U \cap E)^{\frac R {1-\epsilon}}]{
					\frac{|\nabla \hat f (\tilde \xi) \cdot h|^p}{|h|^p}
					\rho_\sigma((1-\epsilon)|h|)
				}[\tilde \xi]
			}[t]
		}[h],
	\end{align*}
	where we applied Fubini's theorem in the last step.
	Choosing a unit vector $e \in \unitsphere n$, which can be thought of as $\frac {\nabla \hat f(\tilde \xi)}{|\nabla \hat f(\tilde \xi)|}$ for all $\tilde \xi$, for which $\nabla \hat f(\tilde \xi) \neq 0$ , we factorize the last expression in our chain of inequalities as
	\begin{align}
		\label{twointegralfactors}
		\frac{(1+\epsilon)^2}{(1-\epsilon)^p} & \left( \integral[\phi(U \cap E)^{\frac R {1-\epsilon}}]{
				|\nabla \hat f (\tilde \xi)|^p
		}[\tilde \xi] \right) \integral[B_{\frac R {1-\epsilon}}^n(0)]{
			\left|e \cdot \frac{h}{|h|} \right|^p \rho_\sigma((1-\epsilon)|h|)
		}[h].
	\end{align}
	We introduce spherical coordinates for $h$ and further rewrite the second integral as
	\begin{align*}
		& \integral[\unitsphere n]{|e \cdot u|^p}[\haumeas{n-1}(u)] \cdot
		\integral[0][\frac R {1-\epsilon}]{\rho_\sigma((1-\epsilon)r) r^{n-1}}[r]  \\
		& = \haumeas{n-1}(\unitsphere n) K_{p,n} (1-\epsilon)^{-n} \integral[0][R]{\rho_\sigma(r) r^{n-1}}[r] \le (1-\epsilon)^{-n} K_{p,n},
	\end{align*}
	since $\integral[0][\infty]{ \rho_\sigma(r) r^{n-1}}[r] = \frac 1 {\haumeas{n-1}(\unitsphere n)}$.
	
	Finally, we transform the integration in $\tilde \xi$ in (\ref{twointegralfactors}) back to an integral over a subset of $M$: 
	The equivalence of Euclidean and geodesic distance (\ref{length}) on one hand implies
	\begin{equation*}
		\phi^{-1}(\phi(U \cap E)^{\frac R {1-\epsilon}}) \subseteq (U \cap E)^{\frac{1+\epsilon}{1-\epsilon}R}\subseteq (U\cap E)^{2R},
	\end{equation*}
	and we choose $R > 0$ in such way, that $4 R < \tau_0$ in condition 2 of Lemma \ref{coveringlemma} (the factor 2 ensures the validity of equation (\ref{onesetestimate}), where $U_k$ is replaced by $U_k^{2R}$, which we  need later).
	On the other hand condition 3 of the same Lemma assures that $|\nabla \hat f(\tilde \xi)| = \left|d\phi|_x (\grad f(x))\right| \le (1+\epsilon) |\grad f(x)|_g$, where $\phi(x) = \tilde \xi$, so using (\ref{detg})
	\begin{align*}
		\integral[\phi(U \cap E)^{\frac R {1-\epsilon}}]{
				|\nabla \hat f (\tilde \xi)|^p
			}[\tilde \xi] \le \frac{(1+\epsilon)^p}{1-\epsilon} \integral[(U \cap E)^{2R}]{
				|\grad f(x)|_g^p
			}[V_g(x)].
	\end{align*}
	After reintroducing the index $k$ the inequality we have proved so far reads as
	\begin{align}
		\label{onesetestimate}
		\integral[U_k \cap E]{
			\integral[B_R^M(x)]{
				 \frac{|f(x)-f(y)|^p}{d(x,y)^p} \rho_\sigma(d(x,y))
			}[V_g(y)]
		}[V_g(x)] \le \frac{(1+\epsilon)^{p+2}}{(1-\epsilon)^{p+n+1}} K_{p,n} \integral[(U_k \cap E)^{2R}]{
			|\grad f(x)|_g^p
		}[V_g(x)].
	\end{align}

	By Propositions \ref{sobapprox} and \ref{bvapprox}, as well as Fatou's lemma, this inequality holds true for all $f \in W^{1,p}(M)$ if $p > 1$ or $f \in BV(M)$ if $p=1$, respectively, where in the latter case $\integral[(U_k \cap E)^{2R}]{
			|\grad f(x)|_g^p
		}[V_g(x)]$ needs to be replaced by $|Df|(U_k \cap E^{2R})$. 
		First, assume that $p > 1$ and $f \in W^{1,p}(M)$:

	The domain of integration $(U_k \cap E)^{2R}$ in \eqref{onesetestimate} is contained in the intersection 
	\begin{equation*}
		U_k^{2R} \cap E^{2R} = (U_k \cap E^{2R}) \cup ((U_k^{2R} \setdiff U_k) \cap E^{2R}),
\end{equation*}
so if we sum up over all $k$ and note that the $U_k$ cover $E$ up to a set of measure zero by Lemma \ref{coveringlemma}, 4., we have
	\begin{align*}
		& \integral[E]{
			\integral[B_R^M(x)]{
				 \frac{|f(x)-f(y)|^p}{d(x,y)^p} \rho_\sigma(d(x,y))
			}[V_g(y)]
		}[V_g(x)] \\
		\le & \frac{(1+\epsilon)^{p+2}}{(1-\epsilon)^{p+n+1}} K_{p,n} \left( \sum_{k = 1}^N \integral[U_k \cap E^{2R}]{|\grad f|_g^p}[V_g] + \sum_{k=1}^N \integral[U_k^{2R} \setdiff U_k]{|\grad f|_g^p}[V_g] \right) \\
		\le & \frac{(1+\epsilon)^{p+2}}{(1-\epsilon)^{p+n+1}} K_{p,n} \left( \integral[E^{2R}]{|\grad f|_g^p}[V_g] + \sum_{k=1}^N \integral[U_k^{2R} \setdiff U_k]{|\grad f|_g^p}[V_g]\right).
	\end{align*}
	The sets $U_k^{2R} \setdiff U_k$ converge to $\partial U_k$ as $R \to 0$, which by Lemma \ref{coveringlemma}, 4., satisfy $\integral[\partial U_k]{}[V_g] = 0$.
	Thus,  put 
	\begin{equation}
		\label{Gdef}
		G_{\varepsilon}(R)=\frac{(1+\varepsilon)^{p+2}}{(1-\varepsilon)^{p+n+1}} K_{p,n} \left(\sum_{k=1}^N \integral[U_k^{2R} \setdiff U_k]{|\grad f|_g^p}[V_g] \right)\quad.
	\end{equation}
	In the case of $p=1$ and $f \in BV(M)$ all computations up to \eqref{Gdef} carry over verbatim, where all integrals of the form $\integral[A]{|\grad f|_g^p}[V_g]$ need to be replaced by $|Df|(A)$ and we need to apply property 4' of Lemma \ref{coveringlemma} to show that $\lim_{R \to 0} G_\epsilon(R) = 0$.
	\\
	\\
	{\bf Step 2:} An upper estimate for $\displaystyle \integral[E]{ \integral[d(x,y) \ge R]{ \frac{|f(x)-f(y)|^p}{d(x,y)^p} \rho_\sigma(d(x,y)) }[V_g(y)] }[V_g(x)]$.\\ 

	For the remaining region consisting of all pairs $(x,y)$ such that $x \in E$ and $d(x,y) \ge R$ we estimate
	\begin{align*}
		\integral[E]{
			\integral[\compl{B_R^M(x)}]{
				\frac{|f(x)-f(y)|^p}{d(x,y)^p} \rho_\sigma(d(x,y))
			}[V_g(y)]
		}[V_g(x)] \le \frac {2^{p-1}} {R^p} (I_1 + I_2),
	\end{align*}
	where
	\begin{align*}
		I_1 & := \integral[E]{
			|f(x)|^p \integral[\compl{B_R^M(x)}]{
				\rho_\sigma (d(x,y))
			}[V_g(y)]
		}[V_g(x)],\ \text{and}\\
		I_2 & := \integral[E]{
			\integral[\compl{B_R^M(x)}]{
				|f(y)|^p \rho_\sigma(d(x,y))
			}[V_g(y)]
		}[V_g(x)].
	\end{align*}
	By monotonicity of $\rho_\sigma$, we estimate
	\begin{align*}
		I_1 \le \text{Vol}_g (M) \rho_\sigma(R) \norm[L^p(M)]{f}^p,
	\end{align*}
	where $\rho_\sigma(R)$ tends to zero as $\sigma \to 0$.
		
	For $I_2$, we observe that the set $K := \set{d(x,y)}{x \in E, y \in \compl{B_R^M(x)}}$ is closed and therefore compact, such that
	\begin{align*}
		I_2 \le C_\sigma \text{Vol}_g(M) \norm[L^p(M)]{f}^p,
	\end{align*}
	where the sequence $C_\sigma := \sup_{r \in K} \rho_\sigma(r)$ converges to zero by locally uniform convergence.
	
	Therefore, putting $\alpha_\sigma := 2^{p-1}\text{Vol}_g(M)(\rho_\sigma(R) + C_\sigma)$, we have
	\begin{align*}
		\integral[E]{
			\integral[\compl{B_R^M(x)}]{
				\frac{|f(x)-f(y)|^p}{d(x,y)^p} \rho_\sigma(d(x,y))
			}[V_g(y)]
		}[V_g(x)] \le \frac {\alpha_\sigma} {R^p} \norm[L^p(M)]{f}^p.
	\end{align*}
\end{proof}
\\
Analogously to \cite{davila} we have the following result of weak-* convergence of Radon measures:

\begin{theorem}
	\label{weakstarconv}
	If $p > 1$ and $f \in W^{1,p}(M)$, the Radon measures $\mu_{\sigma,p}$ defined in (\ref{musigma}) weakly-* converge to $K_{p,n} |\grad f|_g^p\, \mathrm{d}V_g$ as $\sigma \to 0$.

	If $p = 1$ and $f \in BV(M)$, the measures $\mu_{\sigma,1}$ weakly-* converge to $K_{1,n}|Df|$ as $\sigma \to 0$.
\end{theorem}
\\
\\
\begin{proof}
	Proposition \ref{fixed-i-estimate} shows, that for $p \ge 1$ and every compact set $E \subseteq M$
	\begin{equation*}
		\sup_{0<\sigma<1} \mu_{\sigma,p} (E) < \infty,
	\end{equation*}
	so by weak-* compactness there exists a subsequence $\mu_{\sigma_i,p} =: \mu_{i,p}$ and a limit measure $\mu_p$ such that $\mu_{i,p} \stackrel{i \to \infty}{\to} \mu_p$ with respect to the  weak-* topology.
	We need to show, that for every such subsequence $\mu_p = K_{p,n}\nu_p$ , where the measure $\nu_p$ is defined as
	\begin{equation}
		\label{nudef}
		\nu_p(A) :=
	\begin{cases}
		\integral[A]{|\grad f|_g^p}[V_g], & \text{if}\ p > 1, \\ |Df|(A), & \text{if}\ p=1,
\end{cases}
\end{equation}
for every Borel set $A \subseteq M$.\\
	\clearpage
	{\bf Step 1:} $\mu_p(A) \le K_{p,n} \nu_p(A)$ for every Borel set $A \subseteq M$.\\

	By inner regularity of Radon measures, it suffices to prove the inequality for compact sets $E \subseteq M$.
	We apply Proposition \ref{fixed-i-estimate} with $E$ replaced by $\overline{E^{2R}}$ for $\epsilon > 0$ and $R < R_0$.
	Note that the weak-* convergence of the sequence $(\mu_{i,p})$ implies that $\displaystyle \mu_p( E^{2R}) \le \liminf_{i \to \infty} \mu_{i,p}( E^{2R})$, so we get
	\begin{align*}
		\mu_p(E) & \le \mu_p(E^{2R})  \le \liminf_{i \to \infty} \mu_{i,p}( E^{2R}) \le K_{p,n} \nu_p(E^{4R}) + G_{\varepsilon}(2R).
	\end{align*}
	Letting $R \to 0$ and then $\epsilon \to 0$ we obtain the desired inequality, since by compactness $E^{4R} \to E$ as $R \to 0$.
	\\
	\\
	{\bf Step 2:} $\mu_p(M) \ge K_{p,n} \nu_p(M)$.\\

	This step uses a regularization argument similar to the proofs in \cite{leonispector};
	consider a regularization kernel $\psi \in \compsptc \infty (\Real^n)$ with $\integral[\Real^n]{\psi} = 1$ and $\support \psi \subset B_1^n(0)$, and for $\delta > 0$ set
	\begin{equation*}
		\psi_\delta(x) := \frac 1 {\delta^n} \psi \left(\frac x \delta \right),\ x \in \Real^n.
	\end{equation*}
	For $U \subseteq \Real^n$ open we define the mollification of a function $g \in L_{loc}^1 (U)$ for every $x \in U$ with $d(x, \partial U) > \delta$ by 
	\begin{equation*}
		g_\delta (x) := \integral[\Real^n]{f(x-\zeta) \psi_\delta(\zeta)}[\zeta].
	\end{equation*}
	Note that \(g_{\delta}\) is a \(C^{\infty}\) function. Furthermore, fix $\epsilon \in (0,1)$ and consider a finite family of open sets $(U_k)_{k=1}^N$ and corresponding charts $(V_k,\phi_k)$ as in Lemma \ref{coveringlemma} with $E = M$.
	Then define the functions $f_{k,\delta}: U_k \to \Real$ for $k = 1,\dots,N$ and  $\delta > 0$ sufficiently small by $f_{k,\delta}(x) := (f \circ \phi_k^{-1})_\delta (\phi_k (x))$, i.e.
	\begin{equation}
		\label{convolutionmf}
		f_{k,\delta}(x) = \integral[B_\delta^n]{
			(f \circ \phi_k^{-1}) (\phi_k(x) - \zeta) \psi_\delta(\zeta)
		}[\zeta].
	\end{equation}
	Note that \(f_{k,\delta}\) is defined for every \(x\in U_k\) since by property 2. of Lemma \ref{coveringlemma} the function \(\phi_k\) is defined on an \(U_k^{\tau}\) for some \(\tau>0\). Again, \(f_{k,\delta}\) is a \(C^{\infty}\) function on \(U_k\).
	The following calculations take place in only one $U_k$ for $k$ fixed, so we oppress the index $k$ for the sake of readability.
	We denote the radial mollifiers corresponding to the subsequence $\mu_{i,p}$ by $\rho_i$.
	Putting $\xi := \phi(x)$ and $\eta := \phi(y)$ we estimate
	\begin{align}
		& \label{snormconv} \integral[U]{
		\integral[U]{
			\frac {|f_\delta(x) - f_\delta(y)|^p} {d(x,y)^p} \rho_i(d(x,y))
		}[V_g(y)]
	}[V_g(x)] \\
	& \notag \le \frac{(1+\epsilon)^2}{(1-\epsilon)^p} \integral[\phi(U)]{
		\integral[\phi(U)]{
			\frac {| \integral[B_\delta^n]{((f \circ \phi^{-1})(\xi - \zeta) - (f \circ \phi^{-1})(\eta - \zeta)) \psi_\delta(\zeta)}[\zeta]|^p } {|\xi - \eta|^p} \rho_i ((1-\epsilon)|\xi - \eta|)
		}[\eta]
	}[\xi] \\
	& \notag \le \frac{(1+\epsilon)^2}{(1-\epsilon)^p} \integral[\phi(U)]{
		\integral[\phi(U)]{
			\frac { \integral[B_\delta^n]{|(f \circ \phi^{-1})(\xi - \zeta) - (f \circ \phi^{-1})(\eta - \zeta)|^p \psi_\delta(\zeta)}[\zeta] } {|\xi - \eta|^p} \rho_i ((1-\epsilon)|\xi - \eta|)
		}[\eta]
	}[\xi] \\
	& \notag \le \frac{(1+\epsilon)^2}{(1-\epsilon)^p} \integral[\phi(U)^{\delta}]{
		\integral[\phi(U)^{\delta}]{
			\integral[B_\delta^n]{
				\frac {|(f \circ \phi^{-1})(\xi ) - (f \circ \phi^{-1})(\eta)|^p } {|\xi - \eta|^p}\psi_\delta(\zeta) \rho_i ((1-\epsilon)|\xi - \eta|)
			}[\zeta]
		}[\eta]
	}[\xi] \\
	& \notag = \frac{(1+\epsilon)^2}{(1-\epsilon)^p} \integral[\phi(U)^{\delta}]{
		\integral[\phi(U)^{\delta}]{
			\frac {|(f \circ \phi^{-1})(\xi ) - (f \circ \phi^{-1})(\eta)|^p } {|\xi - \eta|^p} \rho_i ((1-\epsilon)|\xi - \eta|)	
		}[\eta]
	}[\xi]\integral[B_\delta^n]{\psi_\delta(\zeta)}[\zeta] \\
	& \notag = \frac{(1+\epsilon)^2}{(1-\epsilon)^p} \integral[\phi(U)^\delta]{
		\integral[\phi(U)^\delta]{
				\frac {|(f \circ \phi^{-1})(\xi) - (f \circ \phi^{-1})(\eta)|^p } {|\xi - \eta|^p} \rho_i ((1-\epsilon)|\xi - \eta|)
		}[\eta]
	}[\xi]\\
	& \notag \le \frac{(1+\epsilon)^{p+2}}{(1-\epsilon)^{p+2}}  \integral[U^{(1+\epsilon)\delta}]{
		\integral[U^{(1+\epsilon)\delta}]{
			\frac{|f(x) - f(y)|^p} {d(x,y)^p} \rho_i \left( \frac{1-\epsilon}{1+\epsilon} d(x,y) \right)
		}[V_g(y)]
	}[V_g(x)] \\
	& \label{regcontraction} \le \left( \frac{1+\epsilon}{1-\epsilon} \right)^{p+2} \integral[U^{(1+\epsilon)\delta}]{
		\integral[M]{
			\frac{|f(x) - f(y)|^p} {d(x,y)^p} \rho_i \left( \frac{1-\epsilon}{1+\epsilon} d(x,y) \right)
		}[V_g(y)]
	}[V_g(x)] 
\end{align}
\begin{flalign}
	& \notag = \left( \frac{1+\epsilon}{1-\epsilon} \right)^{-n+p+2} \integral[U^{(1+\epsilon)\delta}]{
		\integral[M]{
			\frac{|f(x) - f(y)|^p} {d(x,y)^p} \tilde \rho_i(d(x,y))
		}[V_g(y)]
	}[V_g(x)] \\
	& \label{upperboundmu} = \left( \frac{1+\epsilon}{1-\epsilon} \right)^{-n+p+2} \left(\tilde \mu_{i,p}(U) + \tilde \mu_{i,p}\left(U^{(1+\epsilon)\delta} \setdiff U\right)\right),
	\end{flalign}
	where $\tilde \rho_i(r) := \left( \frac {1-\epsilon}{1+\epsilon} \right)^n \rho_i \left( \frac{1-\epsilon}{1+\epsilon} r \right)$ and $\tilde \mu_{i,p}$ is the measure defined by replacing $\rho_i$ with $\tilde \rho_i$ in (\ref{musigma}).

	On the other hand \eqref{snormconv} can be estimated from below via
	\begin{align}
	& \notag \integral[U]{
		\integral[U]{
			\frac {|f_\delta(x) - f_\delta(y)|^p} {d(x,y)^p} \rho_i(d(x,y))
		}[V_g(y)]
	}[V_g(x)] \\
	& \label{convlowbound} \ge \frac{1-\epsilon}{(1+\epsilon)^p} \integral[\phi(U)]{
		\sqrt{\det (\gabhat(\xi))}
		\integral[\phi(U)]{
			\frac {|(f \circ \phi^{-1})_\delta (\xi) - (f \circ \phi^{-1})_\delta (\eta)|^p} {|\xi - \eta|^p} \rho_i((1+\epsilon)|\xi - \eta|)
		}[\eta]
	}[\xi],
	\end{align}
	where the inner integral converges to
	\begin{equation*}
		(1+\epsilon)^{-n} K_{p,n} |\nabla (f \circ \phi^{-1})_\delta (\xi)|^p
	\end{equation*}
	as $i \to \infty$, see \cite[(6)]{bbm}.
	Since the integrand of the outer integral can be estimated by Lipschitz continuity of $(f \circ \phi^{-1})_\delta$ with Lipschitz constant $L_{\delta} > 0$ via
	\begin{align*}
		& \sqrt{ \det(\gabhat(\xi)) } \integral[\phi(U)]{
			\frac {|(f \circ \phi^{-1})_\delta (\xi) - (f \circ \phi^{-1})_\delta (\eta)|^p} {|\xi - \eta|^p} \rho_i((1+\epsilon)|\xi - \eta|)
		}[\eta] \\
		& \le \sqrt{ \det(\gabhat(\xi)) } \integral[\phi(U)]{
			\frac {L_{\delta}^p |\xi - \eta|^p}{|\xi - \eta|^p} \rho_i( (1+\epsilon) |\xi - \eta|)
		}[\eta] \le L_{\delta}^p \sqrt{ \det(\gabhat(\xi)) },
	\end{align*}
	we can apply the dominated convergence theorem for the $\xi$-integration in \eqref{convlowbound}.

	Now we put the estimates \eqref{upperboundmu} and \eqref{convlowbound} together:
	\begin{align}
		& \notag \frac{1-\epsilon}{(1+\epsilon)^p} \integral[\phi(U)]{
		\sqrt{\det (\gabhat(\xi))}
		\integral[\phi(U)]{
			\frac {|(f \circ \phi^{-1})_\delta (\xi) - (f \circ \phi^{-1})_\delta (\eta)|^p} {|\xi - \eta|^p} \rho_i((1+\epsilon)|\xi - \eta|)
		}[\eta]
	}[\xi] \\
	& \label{lowerandupperest} \le \left( \frac{1+\epsilon}{1-\epsilon} \right)^{-n+p+2} (\tilde \mu_{i,p}(U) + \tilde \mu_{i,p}(U^{(1+\epsilon)\delta} \setdiff U)).
	\end{align}

	We claim that	
	\begin{equation}
		\label{differentrhos}
		\limsup_{i \to \infty} \tilde \mu_{i,p} (\overline U) = \limsup_{i \to \infty} \mu_{i,p} (\overline U) + o_\epsilon,
	\end{equation}
	where $o_\epsilon \to 0$ as $\epsilon \to 0$.
	First observe that the sequence $(\tilde \rho_i)_{i \in \Natural}$ is a sequence of radial mollifiers (for $i \to \infty$) itself, such that for $f \in C^1 (M)$ we can repeat the calculations in the proof of Proposition \ref{fixed-i-estimate}, but rather than using one mollifier, we plug in the difference $\rho_i \left(\frac{1-\epsilon}{1+\epsilon} d(x,y) \right) - \rho_i(d(x,y))$, which is non-negative by monotonicity of $\rho_i$, instead:
	\begin{align*}
		 & \integral[\overline U]{
			\integral[B_R^M(x)]{
				\frac{|f(x)-f(y)|^p}{d(x,y)^p} \left( \rho_i \left( \frac{1-\epsilon}{1+\epsilon}d(x,y) \right) - \rho_i(d(x,y)) \right)
			}[V_g(y)]
		}[V_g(x)]\\
		& \le (1+\epsilon)^2 \integral[\phi(\overline U)]{ 
			\integral[B_{\frac R {1-\epsilon}}^n]{
				\frac{|\hat f(\xi) - \hat f(\eta)|^p}{(1-\epsilon)^p|\xi - \eta|^p}
				\left( \rho_i \left(\frac{(1-\epsilon)^2}{1+\epsilon}|\xi-\eta| \right) - \rho_i( (1+\epsilon)|\xi-\eta|) \right)
			}[\eta]
		}[\xi],
	\end{align*}
	where on the right-hand side we used the equivalence of distances \eqref{length} accordingly.
	Following the proof of Proposition \ref{fixed-i-estimate} up to \eqref{twointegralfactors} with the obvious modifications, we see that the last expression does not exceed
	\begin{align}
		\label{rhodiff}
		(1 + o_\epsilon) K_{p,n} \nu_p(U^{2R}) \integral[0][\frac{R}{1-\epsilon}]{\left( \rho_i \left( \frac{(1-\epsilon)^2}{1+\epsilon} r \right) - \rho_i((1+\epsilon)r) \right)r^{n-1}}[r],
	\end{align}
	and this estimate from above still holds true for $f \in W^{1,p}(M)$ or $f \in BV(M)$, if $p=1$, respectively, as can been seen by approximation.
	As $i \to \infty$ \eqref{rhodiff} converges to a remainder $o_\epsilon$, which is 0 as $\epsilon \to 0$.
	The integral over the remaining domain, consisting of all pairs $x \in \overline U, y \notin B_R^M(x)$, is zero in the limit, which we already have seen in Step 2 in the proof of Proposition \ref{fixed-i-estimate}, thus verifying \eqref{differentrhos}.

	Applying the limit in \eqref{lowerandupperest}, and noting that by weak-* convergence $\displaystyle \limsup_{i \to \infty} \mu_{i,p} (\overline U) \le \mu_p (\overline U)$, we obtain
	\begin{align}
		\label{detintegral}
		K_{p,n} \integral[\phi(U)]{
			\sqrt{\det(\gabhat(\xi))} |\nabla (f \circ \phi^{-1})_\delta (\xi)|^p
		}[\xi] \le (1+o_\epsilon) (\mu_p (\overline U) + \mu_p (\overline{U^{(1+\epsilon)\delta} \setdiff U})).
	\end{align}
	Now weed need to distinguish, whether $f \in W^{1,p}(M)$ or $f \in BV(M)$:

	First, let $p > 1$ and $f \in W^{1,p}(M)$.
	Since $|\nabla (f \circ \phi^{-1})_\delta|$ tends to $|\nabla (f \circ \phi^{-1})|$ in $L^p(\phi(U))$ as $\delta \to 0$, and $|\nabla (f \circ \phi^{-1}) (\xi)|^p \ge \frac 1 {(1+\epsilon)^p} |\grad f(\phi^{-1}(\xi))|_g^p$ by Lemma \ref{diffchart}, we have that
	\begin{align*}
		K_{p,n} \integral[U]{|\grad f|_g^p}[V_g] \le (1 + o_\epsilon) (\mu_p(\overline U) + \mu_p(\partial U)).
	\end{align*}

	If $p = 1$ and $f \in BV(M)$, by \eqref{weightedvarbyvar} the integral on the left-hand side of \eqref{detintegral} is equal to the weighted variation $|D(f \circ \phi^{-1})_\delta|_w (\phi(U))$ with weight $w(\xi) = \sqrt{\det(\gabhat(\xi))}$.
	The convolutions $(f \circ \phi^{-1})_\delta$ converge in $L^1 (\phi(U))$ to the function $f \circ \phi^{-1}$, and furthermore
	\begin{align*}
		\integral[\phi(U)]{
			|(f \circ \phi^{-1}) - (f \circ \phi^{-1})_\delta| \sqrt{\det(\gabhat)}
		} \le (1 + \epsilon) \integral[\phi(U)]{
			|(f \circ \phi^{-1}) - (f \circ \phi^{-1})_\delta|
		} \stackrel {\delta \to 0} \to 0.
	\end{align*}
	Since the map $u \mapsto |Du|_w (\phi(U))$ is lower semicontinuous with respect to convergence in $L^1(\phi(U),w\, \mathrm{d}x)$, letting $\delta \to 0$ we obtain
	\begin{align*}
		K_{1,n} |D(f \circ \phi^{-1})|_w (\phi(U)) \le (1 + o_\epsilon) (\mu_1(\overline U) + \mu_1(\partial U)).
	\end{align*}
	By Lemma \ref{diffchart} the left-hand side can further be estimated by $\frac 1 {(1+\epsilon)} K_{1,n}|Df|(U)$ from below.

	For both cases $p > 1$ and $p = 1$ the resulting inequality can be written as
	\begin{equation*}
		K_{p,n} \nu_p(U) \le (1 + o_\epsilon)(\mu_p(\overline U) + \mu_p(\partial U)).
	\end{equation*}
	By our assumptions 4 and 4' of Lemma \ref{coveringlemma} on the mass on the boundary of $U$, Step 1 guarantees $\mu_p(\partial U) \le K_{p,n} \nu_p(\partial U) = 0$, and in consequence
	\begin{equation*}
		K_{p,n} \nu_p(U) \le (1 + o_\epsilon) \mu_p (U).
	\end{equation*}
	Summing up over all $k$ and letting $\epsilon \to 0$ yields the desired inequality.\\
	\\
	\noindent {\bf Step 3:} $\mu_p(A) \ge K_{p,n}\nu_p(A)$ for every Borel set $A \subseteq M$.\\
	\\
	Since $\mu_p$ is a finite measure, for each Borel set $A \subseteq M$ it holds that
	\begin{equation*}
		\mu_p(A) = \mu_p(M) - \mu_p(\compl A) \ge K_{p,n}\nu_p(M) - K_{p,n}\nu_p(\compl A) = K_{p,n}\nu_p(A)
	\end{equation*}
	by the preceding steps 1 and 2.
\end{proof}
\\
\\
With the weak-* convergence at hand, the proof of Theorem \ref{convergencerho} is not difficult anymore:\\

\begin{proof}[Theorem \ref{convergencerho}]
	First, suppose that $f \in W^{1,p}(M)$, if $p >1$, and $f \in BV(M)$, if $p = 1$.
	Since $M$ is both open and compact, the weak-* convergence of $\mu_{\sigma,p}$ to $K_{p,n}\nu_p$ (with $\nu_p$ defined in \eqref{nudef}), which is established in Theorem \ref{weakstarconv}, implies
	\begin{equation*}
		K_{p,n} \nu_p(M) \le \liminf_{\sigma \to 0} \mu_{\sigma,p}(M) \le \limsup_{\sigma \to 0} \mu_{\sigma,p}(M) \le K_{p,n}\nu_p(M),
	\end{equation*}
	which is the desired result.

	On the other hand, suppose that
	\begin{equation}
		\label{limsupfinite}
		\liminf_{\sigma \to 0} \integral[M]{
			\integral[M]{
				\frac{|f(x)-f(y)|^p}{d(x,y)^p} \rho_\sigma(d(x,y))
			}[V_g(y)]
		}[V_g(x)] < \infty.
	\end{equation}
	We show that $f \in W^{1,p}(M)$ and, if $p=1$, that $f \in BV(M)$. By Propositions \ref{sobapprox} and \ref{bvapprox} it is enough to construct a family $(f_\delta)_{\delta>0}$ of \(C^{\infty}\) functions on $M$, such that $f_\delta \to f$ in $L^p(M)$ as $\delta \to 0$ and
	\begin{equation}
		\label{variationsbound}
		\liminf_{\delta \to 0} \integral[M]{|\grad (f_\delta)|_g^p}[V_g] < \infty\quad.
	\end{equation}
	For $\epsilon \in (0,1)$ introduce the modified metric $\tilde g := \frac{1+\epsilon}{1-\epsilon} g$.
	Note that the corresponding distance function satisfies $\tilde d(x,y) = \frac{1+\epsilon}{1-\epsilon} d(x,y)$ for all $x,y \in M$, and the volume form transforms as $\mathrm{d}V_{\tilde g} = \left( \frac{1+\epsilon}{1-\epsilon} \right)^{\frac n 2} \mathrm{d}V_g$.
	Furthermore, a function on $M$ is of bounded variation with respect to $\tilde g$ if and only if it is with respect to $g$, and the variations coincide up to a factor dependent on $\epsilon$.

	Let $(U_k)_{k=1}^N$ be chosen accordingly to Lemma \ref{coveringlemma} and put \(W_k=U_k^{\tau}\) for some \(\tau<\tau_0\). Then \(W_k\) is a covering of $M$ by open sets, i.e. $M = \bigcup_{k=1}^N W_k$. Let $(\chi_k)_{k=1}^n$ be an underlying smooth partition of unity, i.e. smooth functions $\chi_k: M \to [0,1]$ compactly supported in $W_k$ with $\sum_{k=1}^N \chi_k = 1$.
	If \(\delta>0\) is sufficiently small we are able to define regularization functions  $f_{k,\delta}$ on \(W_k\) according to \eqref{convolutionmf}. Putting $f_\delta := \sum_{k=1}^N \chi_k f_{k,\delta}$ yields a family of \(C^{\infty}\) functions such that $f_{\delta} \to f$ in $L^p(M)$ as $\delta \to 0$.

	We estimate 
	\begin{align}
	& \integral[M]{
		\notag	 \integral[M]{
				 \frac{|f_\delta(x)-f_\delta(y)|^p}{\tilde d(x,y)^p} \rho_\sigma(\tilde d(x,y))
			}[V_{\tilde g}(y)]
	}[V_{\tilde g}(x)]  \\ 
	& \notag \le N^{p-1} \sum_{k=1}^N \left( \integral[W_k]{
			\integral[W_k]{
				\frac{|\chi_k(x) f_{k,\delta}(x)-\chi_k(y)f_{k,\delta}(y)|^p}{\tilde d(x,y)^p} \rho_\sigma(\tilde d(x,y))
			}[V_{\tilde g}(y)]
	}[V_{\tilde g}(x)] \right. \\
	\label{gtildeest} & \qquad \left. +  \integral[M \setdiff W_k]{
			\integral[W_k]{
				\frac{|\chi_k(x) f_{k,\delta}(x)-\chi_k(y)f_{k,\delta}(y)|^p}{\tilde d(x,y)^p} \rho_\sigma(\tilde d(x,y))
			}[V_{\tilde g}(y)]
	}[V_{\tilde g}(x)] \right),
\end{align}
where the integrals over $M \setdiff W_k$ tend to 0 as $\sigma \to 0$, since the support of $\chi_k$ is compact in $W_k$ and therefore $\tilde d(x,y) \ge R > 0$.
The remaining summands can be estimated by
\begin{align*}
	\integral[W_k]{
			\integral[W_k]{
			&	\frac{|\chi_k(x) f_{k,\delta}(x)-\chi_k(y)f_{k,\delta}(y)|^p}{\tilde d(x,y)^p} \rho_\sigma(\tilde d(x,y))
			}[V_{\tilde g}(y)]
	}[V_{\tilde g}(x)] \\
	& \le 2^{p-1} \left( \integral[W_k]{ |f_{k,\delta}(x)|^p
			\integral[W_k]{
				\frac{|\chi_k(x) - \chi_k(y)|^p}{\tilde d(x,y)^p} \rho_\sigma(\tilde d(x,y))
			}[V_{\tilde g}(y)]
	}[V_{\tilde g}(x)] \right. \\
	& \left. \qquad + \integral[W_k]{
			\integral[W_k]{
				\frac{|f_{k,\delta}(x) - f_{k,\delta}(y)|^p}{\tilde d(x,y)^p} \rho_\sigma(\tilde d(x,y))
			}[V_{\tilde g}(y)]
	}[V_{\tilde g}(x)] \right) 
	=: 2^{p-1} (I_{k,1} + I_{k,2}).
	\end{align*}
	Since $\chi_k$ is smooth, the inner integral in $I_{k,1}$ converges to $K_{p,n}|\grad \chi_k(x)|_{\tilde g}^p$ as $\sigma \to 0$, and by dominated convergence we have
	\begin{align*}
		\lim_{\sigma \to 0} I_{k,1} = K_{p,n} \integral[W_k]{
			|f_{k,\delta}(x)|^p |\grad \chi_k(x)|_{\tilde g}^p
		}[V_{\tilde g}(x)] \le C \integral[W_k]{|f_{k,\delta}(x)|^p}[V_{\tilde g}(x)],
	\end{align*}
	where $\displaystyle C := \max_{x \in \support\! \chi_k} |\grad \chi_k(x)|_{\tilde g}$.
	By the $L^p$-convergence of $f_{k,\delta}$ as $\delta \to 0$, we furthermore get that $\displaystyle \lim_{\sigma \to 0} I_{k,1}$ is uniformly bounded in $\delta$.

	For the second integrals $I_{k,2}$ we can repeat the calculations of \eqref{snormconv} up to \eqref{regcontraction}, leading to
	\begin{align*}
		I_{k,2} & \le (1+o_\epsilon) \integral[M]{
			\integral[M]{
				\frac {|f(x)-f(y)|^p}{\tilde d(x,y)^p} \rho_\sigma \left( \frac{1-\epsilon}{1+\epsilon} \tilde d(x,y) \right)
				}[V_{\tilde g}(y)]
			}[V_{\tilde g}(x)] \\
			& = (1+o_\epsilon) \integral[M]{
				\integral[M]{
					\frac{|f(x)-f(y)|^p}{d(x,y)^p} \rho_\sigma(d(x,y))
				}[V_g(y)]
			}[V_g(x)],
	\end{align*}
	where in the last line we switched back to the metric $g$ and absorbed the occurring factors into $o_\epsilon$, which converges to 0 as $\epsilon \to 0$.
	Thus, the limit superior of $I_{k,2}$ as $\sigma \to 0$ is finite by our assumption \eqref{limsupfinite}, and even uniformly bounded in $\delta > 0$.

	We conclude by observing that taking the limit $\sigma \to 0$ in the left-hand side of \eqref{gtildeest} yields \newline $K_{p,n} \integral[M]{|\grad f_\delta(x)|_{\tilde g}^p}[V_{\tilde g}]$, thus passing to the original metric, we have showed \eqref{variationsbound}.
\end{proof}
\\

Using suitable radial mollifiers leads to the $s$-seminorm and thus to Corollary \ref{convergencesemi}:\\
\\
\begin{proof}[Corollary \ref{convergencesemi}]
	Define radial mollifiers $\rho_\sigma, \sigma > 0$, by
\begin{equation*}
	\rho_\sigma (r) := \begin{cases}
		\frac{\sigma p}{\haumeas{n-1}(\unitsphere n)} \frac 1 {r^{n-\sigma p}}, & 0 < r < 1, \\
		0, & r \ge 1.
	\end{cases}
\end{equation*}
and set $s := 1 - \sigma$.
We claim, that
\begin{equation*}
	\lim_{\sigma \to 0} \integral[M]{
			\integral[M]{
				\frac{|f(x)-f(y)|^p}{d(x,y)^p} \rho_\sigma(d(x,y))
			}[V_g(y)]
		}[V_g(x)] = \lim_{s \to 1^-} \frac{(1-s)p}{\haumeas{n-1}(\unitsphere n)} \integral[M]{
		\integral[M]{
			\frac{|f(x)-f(y)|^p}{d(x,y)^{n+sp}}
		}[V_g(y)]
	}[V_g(x)],
\end{equation*}
where by Theorem \ref{convergencerho} the left-hand side is equal to either $K_{p,n} \integral[M]{|\grad f|_g^p}[V_g]$, if $p>1$, or $K_{1,n}|Df|(M)$, if $p=1$.
To see this, we only need to show that
\begin{equation*}
	\lim_{s \to 1^-} (1-s) \integral[M]{
		\integral[\simpleset{y : d(x,y) \ge 1}]{
			\frac{|f(x)-f(y)|^p}{d(x,y)^{n+sp}}
		}[V_g(y)]
	}[V_g(x)] = 0.
\end{equation*}
But this is a simple consequence of
\begin{align*}
	\integral[M]{
		\integral[\simpleset{y : d(x,y) \ge 1}]{
			\frac{|f(x)-f(y)|^p}{d(x,y)^{n+sp}}
		}[V_g(y)]
	}[V_g(x)] & \le  \integral[M]{
		\integral[M]{
			|f(x)-f(y)|^p
		}[V_g(y)]
	}[V_g(x)] \\
	& \le 2^p \text{Vol}_g(M) \norm[L^p]{f}^p.
\end{align*}
\end{proof}

\subsection*{Acknowledgement}

The authors would like to thank Monika Ludwig for helpful comments and suggestions during the preparation of this paper.


\bibliographystyle{plainnoand}
\bibliography{fracmf}

\begin{minipage}[t]{.46\textwidth}
	Andreas Kreuml\\
	{\small
	Institut f\"ur Diskrete Mathematik und Geometrie\\
	Technische Universit\"at Wien\\
	Wiedner Hauptstra{\ss}e 8-10/1046\\
	1040 Vienna, Austria\\
E-mail: andreas.kreuml@tuwien.ac.at}
\end{minipage}
\hfill~\hfill
\begin{minipage}[t]{.46\textwidth}
	Olaf Mordhorst\\
	{\small
	Institut f\"ur Diskrete Mathematik und Geometrie\\
	Technische Universit\"at Wien\\
	Wiedner Hauptstra{\ss}e 8-10/1046\\
	1040 Vienna, Austria\\
E-mail: olaf.mordhorst@tuwien.ac.at}
\end{minipage}

\end{document}